\documentclass[draft,reqno,11pt]{amsart}
 
\usepackage{amsmath,amssymb,amsthm, epsfig}
\usepackage[usenames,dvipsnames]{pstricks} 
\usepackage{epsfig}
\usepackage{pst-grad}
\usepackage{pst-plot}

 \usepackage[usenames,dvipsnames]{pstricks}
 \usepackage{epsfig} 
 \usepackage{pst-grad} 
 \usepackage{pst-plot} 
 \usepackage[space]{grffile} 
 \usepackage{etoolbox} 
 \makeatletter 
 \patchcmd\Gread@eps{\@inputcheck#1}{\@inputcheck"#1"\relax}{}{}
 \makeatother

\newtheorem{theorem}{Theorem} 
\newtheorem{lemma}{Lemma} 
\newtheorem{proposition}{Proposition}  

\newtheorem{definition}{Definition}
\newtheorem{remark}{Remark}

\numberwithin{equation}{section}
\title[Higher regularity estimates for the porous medium equations]{Higher regularity estimates for the porous medium equation near the Heat equation}
 
\author[D.J. Ara\'ujo]{Dami\~ao J. Ara\'ujo}
\address{Department of Mathematics \\ Universidade Federal da Para\'iba \\ 58059-900 \\ Jo\~ao Pessoa-PB, Brazil}{}
\email{araujo@mat.ufpb.br}

\begin{document}



\begin{abstract} 
In this paper we investigate regularity aspects for solutions of the nonlinear parabolic equation 
\begin{equation}\nonumber 
u_t= \Delta u^m, \quad m > 1 
\end{equation}  
usually called the porous medium equation. More precisely, we provide sharp regularity estimates for bounded nonnegative weak solutions along the free boundary $\partial\{u>0\}$, when the equation is universally close to the heat equation.  As a consequence, local Lipschitz estimates are also established for this scenario.
\end{abstract}   

\date{\today}

\maketitle


\section{Introduction} \label{sct intro}

The aim of this article is to study fine regularity properties for solutions of the \textit{porous medium equation}-(pme) 
\begin{equation}\label{m-eq}\tag{$m$-pme} 
u_t= \Delta u^m, \quad m > 1. 
\end{equation} 

The mathematical analysis involving \textit{pme} has attracted attention for the last six decades, motivated by its relation with natural phenomena models which describe processes involving fluid flow, heat transfer and, in general terms, nonlinear diffusion process, cf. \cite{vasquez}.
 
Unlike the uniform parabolicity exhibited for the classical linear \textit{heat equation} $u_t=\Delta u$, case $m=1$, for parameters $m > 1$ the finite speed of propagation property holds for \eqref{m-eq} and so, solutions may present parabolic degeneracy along the set 
$$
\mathcal{F}(u)=\partial\{(t,x) \; : u(t,x) \neq 0\},
$$
which imposes lack of smoothness for solutions. In general, the best regularity result known guarantees local $C^{0,\alpha}$ regularity for solutions in time and space for some universal $0<\alpha \ll 1$, see \cite{D,DF1}. However, no more information is known for the exponent $\alpha$. In this connection, specifically for dimension $d=1$, it was established that the pressure term $\varrho \approx u^{m-1}$ is locally Lipschitz continuous in space, see \cite{A1,AC1}. From this fact, solutions are $C^{\,0,\beta}$ for the sharp exponent  
$$ 
\beta=\min\{1,1/(m-1)\}.
$$
Nevertheless, it does not occur in higher dimensions due to a counter-example provided in \cite{AG}. On the other hand, under natural extra conditions, solutions may present surprising gains of regularity. For instance, it was shown in \cite{KKV} that after a time interval locally flat solutions of \eqref{m-eq} and their respective pressure terms are $C^\infty$. This exemplify how the question of obtaining sharp regularity estimates for solutions of \eqref{m-eq} has been an interesting and delicate subject. We also mention \cite{AMU} for sharp regularity estimates, obtained in terms of the integrability of the source term and \cite{CVW} for Lipschitz regularity estimates for large times. 

In the 1950s,  a fundamental solution for  \eqref{m-eq} was found by Barenblatt \cite{barenblatt}, Zel'dovich and Kompaneets \cite{ZK} and later by Pattle \cite{pattle}. They obtained the following explicit formula 
$$
\mathcal{B}_m(t,|x|,C) =t^{-\alpha} \left(C - \frac{b(m-1)}{2m}\frac{|x|^2}{t^{2b}}\right)_+^{\frac{1}{m-1}}  
$$
for a free parameter $C>0$, $\alpha=\frac{d}{d(m-1)+2}$ and $b=\frac{\alpha}{d}$, where $d$ denotes the dimension of the euclidean space.  This fundamental solution, also called \textit{Barenblatt solution}, presents a Dirac mass as the initial data due to $\mathcal{B}_m(t,|x|,C) \to M \delta_0(x)$ as $t \to 0$, where $M(C,m,d)=\int\mathcal{B}_m \,dx$ is the total mass. Based on the structure of $\mathcal{B}_m$, we observe that at the inner edge 
$$ 
\mathcal{F}(\mathcal{B}_m):=\partial\left\{(t,x) \, : \,\mathcal{B}_m(t,|x|,C)>0\right\}=\left\{(t,x) \, : \, |x|=\sqrt{\frac{2Cm}{b(m-1)}}\, \cdot t^{b} \right\}, 
$$
the gradient blows up for $m > 2$, is finite when $m = 2$, and vanishes (but with a nonzero derivative) in the case $1<m<2$. More precisely, we note that 
$$
\lim\limits_{m\to 1}\mathcal{B}_m(t,|x|,C) = M\mathcal{E}(t,x)
$$ 
where $\mathcal{E}(t,x) \in C^\infty$ is the fundamental solution of the heat equation. Therefore, even under a rough initial data, we observe that for any $t>0$, the family $\{\mathcal{B}_m\}_{m>1}$ plays the following role: the smoothness of $\mathcal{B}_m$ at $\mathcal{F}(\mathcal{B}_m)$ increases asymptotically to, let us say, $C^\infty$ as $m\searrow 1$.

\begin{figure}[h!]
\centering
%
\psscalebox{0.75 0.75} 
{
\begin{pspicture}(0,-4.577778)(15.244445,4.577778)
\definecolor{colour0}{rgb}{0.77254903,0.75686276,0.75686276}
\definecolor{colour1}{rgb}{0.56078434,0.5529412,0.5529412}
\definecolor{colour2}{rgb}{0.56078434,0.53333336,0.53333336}
\psline[linecolor=black, linewidth=0.04, arrowsize=0.05291667cm 2.0,arrowlength=1.4,arrowinset=0.0]{->}(0.12222222,-2.5555556)(15.1,-2.5666666)
\rput[bl](0.4888889,3.9355555){$t>0$}
\rput[bl](14.666667,-2.9444444){$x$}
\rput[bl](7.411111,-0.5888889){$\mathcal{E}$}
\psbezier[linecolor=colour0, linewidth=0.04](5.111111,-2.5666666)(5.5111113,0.23333333)(6.711111,3.8333333)(7.5111113,3.833333333333332)(8.311111,3.8333333)(9.511111,0.23333333)(9.911111,-2.5666666)
\rput[bl](7.311111,3.1222222){$\mathcal{B}_2$}
\psbezier[linecolor=colour1, linewidth=0.04](3.488889,-2.5555556)(5.9666667,-0.54888886)(6.611696,2.2777777)(7.4888887,2.311111111111107)(8.366082,2.3444445)(9.733334,-1.48)(11.488889,-2.5555556)
\rput[bl](7.233333,1.5){$\mathcal{B}_m$}
\psbezier[linecolor=black, linewidth=0.04](0.23333333,-2.3333333)(5.311111,-2.211111)(6.532406,-0.06666667)(7.577778,-0.06666666666666572)(8.62315,-0.06666667)(10.244445,-2.4333334)(14.866667,-2.3333333)
\psdots[linecolor=black, dotsize=0.1](9.911111,-2.5666666)
\psdots[linecolor=black, dotsize=0.1](11.488889,-2.5777779)
\psdots[linecolor=black, dotsize=0.1](5.1,-2.5666666)
\psdots[linecolor=black, dotsize=0.1](3.5,-2.5555556)
\psframe[linecolor=black, linewidth=0.02, dimen=outer](15.244445,4.577778)(0.0,-4.577778) 
\rput[bl](0.5,-4.308889){$m\searrow 1$}
\rput[bl](9.444445,-3.2222223){$\mathcal{F}(\mathcal{B}_2)$}
\rput[bl](10.966666,-3.2222223){$\mathcal{F}(\mathcal{B}_m)$} 
\rput[bl](10.433333,-4.3){$\{\mathcal{B}_m>0\} \to  \{\mathcal{E}>0\}= \mathbb{R}^d$}
\psline[linecolor=colour2, linewidth=0.02, tbarsize=0.07055555cm 5.0]{|*-|*}(3.5,-3.0333333)(7.577778,-3.0333333)
\rput[bl](4.6,-3.8755555){$\scriptsize O(\frac{1}{\sqrt{m-1}})$} 
\end{pspicture}
} 
\caption{This picture represents the improvement of regularity for the Barenblatt solution as $m \searrow 1$: around the free boundary $\mathcal{F}(\mathcal{B}_m)$, it describes a surface leading towards a smooth surface.} 
\end{figure}
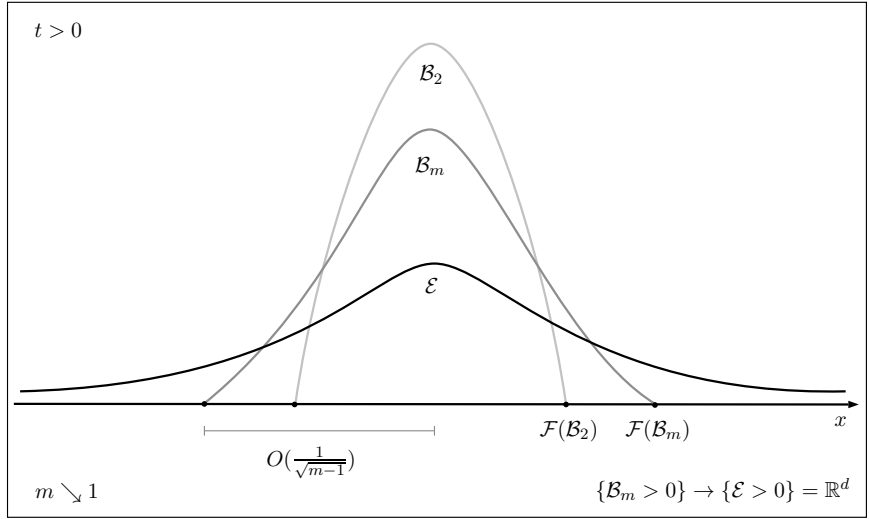

Motivated by such analysis, we turn our attention to investigate high regularity, in space and time, along interior points of the free boundary $\mathcal{F}(u):=\partial\{u>0\}$, for nonnegative bounded weak solutions of \eqref{m-eq} as $m$ is universally close to $1$. More precisely, fixed any parameter $\mu \in (0,+\infty)$ we provide an universal closeness regime such that solutions of \eqref{m-eq} are pointwisely of the class $C^{\mu}$ at $\mathcal{F}(u)$. 

Although such features seen to be appropriate for this scenario, it presents further difficulties. For example, we observe that \eqref{m-eq} revels a variance on the diffusion velocity, which depends directly on $m$. Thus to gather high regularity from the limit case $m=1$, a suitable strategy is needed since the intrinsic cylinders are changing when $m$ varies, see \eqref{cyl}. Motivated by this, we could try instead of solution $u$, to provide growth estimates through the pressure
\begin{equation}\nonumber
\varrho= c\, u^{m-1}, \quad \quad c > 0,
\end{equation}   
which turns our analysis to the pressure equation
$$
\varrho_t= {\tfrac{m}{c}} \varrho \Delta \varrho + {\tfrac{m}{c(m-1)}} |\nabla \varrho|^2. 
$$ 
Unlike equation \eqref{m-eq} the diffusion term related to the equation above does not depend on $m$ and so, an analysis of such pde would sound reasonable. However, since $m\searrow 1$ the pressure constant has to behave like $c \approx 1/(m-1)$ and so, no further information is available from the second order term of the limit equation.

As an important consequence of such analysis, we also study nonnegative bounded weak solutions for the inhomogeneous porous medium equation with bounded source term,
\begin{equation}\label{mf-eq}\tag{$f$,$m$-pme}
u_t-\Delta u^m = f \in L^\infty,
\end{equation}
providing for $m-1 \ll 1$, optimal growth estimates
$$
u(t,x) \sim |x-x_0|^{\frac{2}{m}} + |t-t_0|
$$
at touching ground points $(t_0,x_0) \in \mathcal{F}(u)$. Related to this scenario, we mention \cite{PS} for improved H\"older regularity estimates at the free boundary, under integrability conditions on the source term.

Futhermore, we also refine the methods employed here to provide sharp local $C^{0,1}$ regularity in space and  $C^{0,\frac 12}$ regularity in time for equations of the type \eqref{mf-eq}. We have postponed the precise statements to Section \ref{main}. 

The main strategy of this paper is based on the Caffarelli's compactness approach \cite{C89}, and a refined improvement of flatness strategy provided by Teixeira in \cite{T2}. For this scenario, we point out that, to produce a suitable flatness property, the compactness argument provides how $m$ has to be universally close to $1$. However, such closeness shall depend on the parabolic metric from which (unlike their elliptic counterpart) varies on $m$, causing a self-dependence on this parameter. In order to avoid this issue, a subtle use of dyadic parabolic cylinders is provided, see Proposition \ref{stepone}.

The paper is organized as follows. In Section \ref{pre} we gather relevant notations and present the main Theorems, as well as known results we shall use in this article. In Section \ref{grow} we treat the homogeneous case and the proof of Theorem \ref{fbreg} is delivered. The inhomogeneous case is discussed in Section \ref{inhomo}, where the proofs of Theorem \ref{fbregnon} and Theorem \ref{locreg} are carried out.

\section{Preliminaries and Main results} \label{pre}

\subsection*{Notations and intrinsic parabolic cylinders}
The lack of homogeneity caused by the nature of degeneracy or singularity of certain parabolic equations requires a refined choice for suitable cylinders, cf. \cite{TU,U,vasquez}. In view of this, for a fixed  open bounded set $U \subset \mathbb{R}^d$ and parameter $\theta>0$, we introduce the intrinsic $\theta$-parabolic cylinder  
\begin{equation}\label{cyl}
G_\rho^\theta :=I_\rho^\theta \times B_\rho \subset \mathbb{R} \times U,
\end{equation}
where $I_\rho^{\,\theta}:=(-\rho^\theta,0 \,]$ is an 1-dimensional interval and $B_\rho$ is the $d$-dimensional ball with radius $\rho$ centered at the origin. More generally, we set $I^{\,\theta}_\rho(t_0):=(-\rho^\theta+t_0,t_0\,]$ and $B_\rho(x_0):=\{x_0\}+B_\rho$.  We also denote $G_1:=G_1^\theta$ for any $\theta$.  
It is easy to check for any number $0<\mu<\infty$, equations of the type \eqref{m-eq} have an invariant scaling under cylinders \eqref{cyl} with the interpolation space/time given precisely by
\begin{equation}\label{interp}
\theta=\theta(\mu,m):=\mu(1-m)+2.
\end{equation}
We highlight the following monotonicity 
$$
G_\rho^{\theta} \subset G_\rho^{\theta'} \quad \mbox{for any } 0<\theta' \leq \theta \mbox{ and } 0<\rho<1.    
$$ 
Next, we introduce the notion of solutions of \eqref{mf-eq} we shall work with. 

\begin{definition}\label{def}
 We say a nonnegative locally bounded function
 $$
 u \in C_{loc}(0,T;L_{loc}^2(U)), \quad u^{\frac{m+1}{2}} \in L_{loc}^2(0,T,W_{loc}^{1,2}(U))
 $$
 is a local weak solution of \eqref{m-eq} if for every compact set $K \subset U$ and every sub interval $[t_1,t_2] \subset (0,T]$, there holds
 $$
 \left.\int_K u\varphi \right\rvert^{t_2}_{t_1}+\int^{t_2}_{t_1}\int_K-u\varphi_t+mu^{m-1}\nabla u\nabla\varphi \, = \int_{t_1}^{t_2}\int_Kf\varphi,
 $$
 for all nonnegative test functions 
 $$
 \varphi \in W^{1,2}_{loc}(0,T;L^2(K)) \cap L^2_{loc}(0,T;W^{1,2}_0(K)).
 $$
\end{definition}

Given a nonnegative function $u$ we denote the parabolic positive set of $u$ and the parabolic free boundary, respectively by
$$
\mathcal{P}(u):=\{(t,x) \in G_1 \, : \, u(t,x)>0\} \quad and \quad \mathcal{F}(u):=\partial \mathcal{P}(u) \cap G_1.
$$  
According to the above definition, we interpret the gradient term as
$$
u^{m-1}\nabla u := \frac{2}{m+1}u^{\frac{m-1}{2}}\nabla u^{\frac{m+1}{2}} \quad \mbox{in} \quad \mathcal{P}(u).  
$$ 

\medskip 

\subsection{Main results}\label{main}

Here we present the main results we shall prove in this paper. The first one provides high regularity estimates for the homogeneous equation \eqref{m-eq} at free boundary points $(t,x) \in \mathcal{F}(u)$.  

\begin{theorem}[High growth estimates at free boundary points]\label{fbreg} 
Fixed $0<\mu < \infty$, there exist parameters $0<m_\mu \leq 1+1/\mu$, $0<\rho_0 < 1/4$ and $C>0$ depending only on $d, \|u\|_{\infty,G_1}$ and $\mu$ such that, for 
$$
1 < m \leq m_\mu,
$$
nonnegative locally bounded weak solutions $u$ of \eqref{m-eq}  in $G_{1}$ satisfy 
$$
\sup\limits_{I^\theta_\rho(t_0)\times B_\rho^{}(x_0)} u(t,x) \leq C \rho^{\,\mu}
$$ 
for each $(t_0,x_0) \in \mathcal{F}(u) \cap G_{\frac{1}{2}}^{\theta}$ and $0<\rho\leq \rho_0$. 
\end{theorem} 

For the second result, we establish the optimal rate growth for solutions of the inhomogeneous case \eqref{mf-eq} at free boundary points $(t,x) \in \mathcal{F}(u)$. Such optimality is observed by considering the stationary solution $u(t,x)=u(x)=C|x|^{2/m}$, for $C>0$ depending on $d$ and $m$.

\begin{theorem}[Optimal growth estimates at free boundary points]\label{fbregnon} 
There exist parameters $\tilde m>1$, $0<\overline\rho\,<1/4$ and $C>0$ depending only on $d$, $\|f\|_{\infty,G_1}$ and $\|u\|_{\infty,G_1}$ such that for each
$$
1 < m \leq \tilde m,
$$
nonnegative locally bounded weak solutions $u$ of \eqref{mf-eq}  in $G_{1}$ satisfy
\begin{equation}\label{fbestnon} 
\sup\limits_{I^{\theta}_\rho(t_0)\times B_\rho^{}(x_0)} u(t,x) \leq C \rho^{\,\frac{2}{m}}
\end{equation}
for each $(t_0,x_0) \in \mathcal{F}(u) \cap G_{\frac{1}{2}}^{\theta}$ and $0<\rho\leq \overline\rho$ with $\theta=\frac{2}{m}$.  
\end{theorem} 

Next, we establish local sharp regularity estimates in space and time, for nonnegative bounded solutions of \eqref{mf-eq}. 

\begin{theorem}[Local sharp regularity]\label{locreg}
There exist parameters $\tilde m>1$, $0<\overline\rho\,<1/4$ and $C>0$ depending only on $d$, $\|f\|_{\infty,G_1}$ and $\|u\|_{\infty,G_1}$ such that for each
$$
1 < m \leq \tilde m, 
$$
nonnegative locally bounded weak solutions $u$ of \eqref{mf-eq}  in $G_{1}$ satisfy
\begin{equation}\label{locest}
\sup\limits_{(t,x) \in I^2_\rho \times B_\rho^{}} |u(t,x)-u(s,y)| \leq C \rho^{}
\end{equation}
for each $0<\rho \leq \overline \rho$ and $(s,y) \in I^2_{1/4} \times B_{1/4}$. In particular, nonnegative locally bounded weak solutions of \eqref{m-eq} are locally of class $C^{\,0,1}$ in space and $C^{\,0,\frac{1}{2}}$ in time.   
\end{theorem}

\medskip

\subsection*{Auxiliary results}

Here we shall mention some important results to be required trough this paper. First, we remark the strong maximum principle for the heat equation, see for instance \cite[Theorem 11]{evans}.

\begin{theorem}\label{Evans} Assume $h(x,t)$ is a weak solution of the heat equation $h_t=\Delta h$ defined in $G_1$. If there exists a point $(t_0,x_0) \in G_1$ such that 
$$  
h(t_0,x_0)=\max\limits_{\overline{G_1}}h(t,x) 
$$
then $u$ is constant everywhere in $(-1,t_0\,] \times B_1$. 
\end{theorem}  

Next, we are going to mention stable local regularity estimates for solutions of \eqref{m-eq} within their domain of definition. As obtained in \cite{DGV},  see theorem 11.2 and remark 1.1, constants $C$ and $\alpha$ in theorem \ref{compactness} below, are stable as $m\searrow 1$.

\begin{theorem}\label{compactness}
Let $1 \leq m \leq 2$ and $u$ be a bounded weak solution of \eqref{m-eq} in $G_1$. Then there exist constants $C>0$ and $0<\alpha<1$ depending only on $d$ and $\|u\|_{\infty, G_1}$ such that
$$
|u(t,x)-u(s,y)| \leq C(|x-y|^\alpha + |t-s|^{\frac{\alpha}{2}})
$$
for any pair of points $(t,x),(s,y) \in (-\frac{1}{2},0] \times B_{\frac{1}{2}}$. 
\end{theorem}

\medskip

\subsection*{Normalization regime} 

For the results to be established in this paper it is enough, with no loss of generality, to consider normalized weak solutions $v$ of \eqref{m-eq} , i.e., satisfying $\|v\|_{\infty,G_1} \leq 1$. Indeed, in the case the results are established for normalized solutions, for any bounded weak solution $u(t,x)$ we may redefine it as follows
\begin{equation}\label{norm}
v(x,t) := \frac{u(N^{b}t,N^ax)}{N} \quad \mbox{in} \quad G_1,
\end{equation}
for $N=\max\{1,\|u\|_{\infty,G_1}\}$ with the following requirements
$$
b=2a-(m-1) \quad \mbox{and} \quad a<0.  
$$
Since $v$ still solves \eqref{m-eq} in $G_1$ satisfying $0\leq v \leq 1$, we conclude that Theorem \ref{fbreg} and Theorem \ref{locreg} can also be obtained for the non-normalized $u(t,x)$ with parameters under the additional dependence on $\|u\|_{\infty,G_1}$. 

By a similar analysis, we have that, for a given universal $\varepsilon_0>0$, we can always enter the smallness regime $\|f\|_{\infty,G_1} \leq \varepsilon_0
$ for a universal rescaled solution of \eqref{mf-eq} as in \eqref{norm}.

\medskip

\section{Proof of Theorem \ref{fbreg}}\label{grow}

The next result provides a universal flatness estimate that allows us to construct a refined decay of solutions in dyadic parabolic cylinders centered at a free boundary point. 

 \begin{lemma}\label{complem} Given $\kappa>0$ there exists $m_\kappa>1$, depending only on $\kappa$ and universal parameters, such that if $(0,0) \in \mathcal{F}(v)$, $0 \leq v \leq 1$ and $v$ satisfies \eqref{m-eq} in $G_1$ for
$$
1 \leq m \leq m_\kappa,
$$
then
$$
\sup \limits_{I_{1/2}^2 \times B_{1/2}^{}} v(t,x) \leq \kappa. 
$$  
\end{lemma} 

\begin{proof}
For the sake of contradiction, we assume the existence of $\kappa_0>0$ and sequences $(v_\iota)_{\iota \in \mathbb{N}}$ and $(m_\iota)_{\iota \in \mathbb{N}}$ where 
$$
v_\iota \in C_{loc}(-1,0;L^2_{loc}(B_1)), \quad (v_\iota)^{\frac{m_\iota+1}{2}} \in L^2(-1,0;W^{1,2}_{loc}(B_1))
$$ 
and $0 \leq v_\iota \leq 1$ such that $v_\iota$ is a nonnegative bounded weak solution of 
\begin{equation}\label{j-eq}\tag{$m_\iota$-pme}
(v_\iota)_t = \Delta(v_\iota^{m_\iota}) \quad \mbox{in} \quad G_1,
\end{equation}
with $m_\iota \to 1$ as $\iota\to \infty$; however 
\begin{equation}\label{comp1} 
\sup \limits_{I_{1/2}^2 \times B_{1/2}^{}} v_\iota(t,x) > \kappa_0.
\end{equation}
By stable H\"older continuity, Theorem \ref{compactness}, the family $(v_\iota)_{\iota \in \mathbb{N}}$ is equicontinous and bounded, therefore 
$$
v_\iota \to \tilde v \quad \mbox{uniformly in }(-1/2,0] \times B_{1/2}, 
$$
where the limiting function $\tilde v \geq 0$ solves
$$
\tilde v_t=\Delta \tilde v \quad \mbox{in } I_{1/2}^2 \times B_{1/2}^{}, 
$$ 
attaining minimum value at $(0,0)$. Therefore, by the strong maximum principle, theorem \ref{Evans}, we must have $\tilde v \equiv 0$ in $G_{1/2}^2$. This contradicts \eqref{comp1} and the proof of Lemma \ref{complem} is complete.
\end{proof}

\begin{proposition}[Improvement of flatness]\label{stepone} Given $0<\mu< \infty$, there exists a parameter $1 < m_\mu \leq 1+1/\mu$, depending only on $\mu$ and universal parameters, such that if $(0,0) \in \mathcal{F}(v)$, $0 \leq v \leq 1$ and $v$ is a bounded weak solution of \eqref{m-eq} in $G_1$ for
$$
1 < m \leq m_\mu,
$$
then
$$
\sup \limits_{I_{\rho}^\theta \times B_{\rho}^{}} v(t,x) \leq \rho^{\,\mu},  \quad \mbox{for} \quad \rho:=\left( \frac12 \right)^{\frac2\theta}.
$$ 
\end{proposition}   

\begin{proof}
First let us fix $0<\mu<\infty$.  Since we assumed $1 < m \leq 1+1/\mu$, it is easy to observe that for the parameter $\theta$ given in \eqref{interp}, there holds $1 \leq \theta < 2$. This implies that
\begin{equation}\label{improv1}
\left(\dfrac{1}{2}\right)^{2} \leq \left(\dfrac{1}{2}\right)^{\frac{2}{\theta}} < \;\, \dfrac{1}{2}.
\end{equation}
Now, considering in Lemma \ref{complem}
$$ 
\kappa=\left(\frac{1}{2}\right)^{2\mu},
$$ 
we guarantee the existence of a parameter $m_\mu$ depending only on $\mu$, such that $1<m_\mu \leq 1+1/\mu$, where for each $1 \leq m \leq m_\mu$ weak solutions of \eqref{m-eq} satisfy
$$
\sup \limits_{I_{1/2}^2 \times B_{1/2}^{}} v(t,x) \leq \left(\frac{1}{2}\right)^{2\mu}. 
$$
In view of this, by choosing $\rho:=\left(\dfrac{1}{2}\right)^{\frac{2}{\theta}}$ we have $I_{1/2}^2=I_{\rho}^\theta$. Then, from \eqref{improv1} 
$$
\sup \limits_{I_{\rho}^\theta \times B_{\rho}^{}} v(t,x) \leq \left(\frac{1}{2}\right)^{2\mu} \leq \left(\frac{1}{2}\right)^{\frac{2\mu}{\theta}}=\rho^{\,\mu},
$$
as desired.
\end{proof}

\begin{proof}[Proof of Theorem \ref{fbreg}]
Up to a translation, we can suppose with no loss of generality that $x_0=0$ and $t_0=0$. Also, from \eqref{norm} we can assume $0 \leq v \leq 1$. By now, we show a discrete version of Theorem \ref{fbreg}. In other words, for $\rho$ and $m_\mu$ as in Proposition \ref{stepone}, there holds 
\begin{equation}\label{kstep}
\sup \limits_{I_{\rho^k}^\theta \times B_{\rho^k}^{}} v(t,x) \leq \rho^{\,k\mu}, \quad k \in \mathbb{N}.
\end{equation}
We prove this inductively.  Note that case $k=1$ is precisely Proposition \ref{stepone}. Let us assume that \eqref{kstep} holds for some positive integer $k>1$. Set
$$
v_k(t,x):=\frac{v(\rho^{\,k\theta}t,\rho^{\,k} x)}{\rho^{\,k\mu}}, \quad (t,x) \in G_1. 
$$
It is easy to check that $v_k$ solves \eqref{m-eq} in $G_1$ with the same exponent $m$ satisfying $1 \leq m \leq m_\mu$ with $\theta$ as in \eqref{interp}. 
In addition, by the induction hypothesis, \eqref{kstep} holds and so we have
$$
\|v_k\|_{\infty, G_1} \leq 1. 
$$
Thus $v_k$ satisfies the hypothesis in Proposition \ref{stepone}. From that, we obtain 
$$
\sup \limits_{I_{\rho}^\theta \times B_{\rho}} v_k(t,x) \leq \rho^{\,\mu},
$$ 
which give us
$$
\sup \limits_{I_{\rho^{\,k+1}}^{\theta} \times B_{\rho^{\,k+1}}^{}} v(t,x) \leq \rho^{\,\mu(k+1)}.
$$

To complete the proof of Theorem \ref{fbreg}, we argue as follows. For a given $0< r \leq 1/4$, we select the positive integer $k_r$ such that 
$$
\rho^{k_r+1}< r \leq \rho^{k_r}.
$$
Finally, remembering that $\rho=\left(\frac{1}{2}\right)^{\frac{2}{\theta}}$, by \eqref{kstep} and \eqref{improv1}, we conclude
$$ 
\sup \limits_{I_{r}^\theta \times B_{r}^{}} v(t,x) \leq \sup \limits_{I_{\rho^{k_r}}^\theta \times B_{\rho^{k_r}}^{}} v(t,x) \leq \rho^{\,-\mu} \rho^{\,(k_r+1)\mu} < 2^{\frac{2\mu}{\theta}} r^{\,\mu} \leq 4^{\mu} r^{\,\mu}.
$$
\end{proof}

\medskip

\section{The inhomogeneous case}\label{inhomo}

In this section we shall deliver the proofs of Theorem \ref{locreg} and Theorem \ref{fbregnon} which provide regularity estimates for nonnegative bounded weak solutions of \eqref{mf-eq} for parameters $m$ universally close to $1$. Hereafter in this section, we assume the intrinsic cylinder exponent $\theta=\theta(\mu,m)$ for the particular case $\mu=2/m$, see \eqref{interp}. In this case such exponent is given by 
$$
\theta=2/m. 
$$  

Now we state the following approximation lemma. 

\begin{lemma}\label{complemma2} Let $v$ be a nonnegative bounded weak solution of \eqref{mf-eq} in $G_1$ with $0 \leq v \leq 1$. Given $\kappa>0$, there exists $\varepsilon>0$, depending on $\kappa$ and $m$ such that if  
\begin{equation}\label{lips1}
v(0,0) + \|f\|_{\infty,G_1}\leq \varepsilon,
\end{equation}
we can find $\varpi$, a bounded weak solution of $\eqref{m-eq}$ in $G^1_{1/2}$, satisfying $\varpi(0,0)=0$, such that
\begin{equation}\label{lips9}
\sup\limits_{I^\theta_{1/2} \times B_{1/2}}|v(t,x)-\varpi(t,x)| \leq \kappa.
\end{equation}
\end{lemma}

\begin{proof} As in the proof of Lemma \ref{complem} we assume, for the sake of contradiction, that there exists $\kappa_0>0$ and, for each positive integer $k$, functions $v_k$ and $f_k$, such that 
$0 \leq v_k \leq 1$, $\|f_k\|_\infty \leq 1/k$ and $v_k$ is a weak solution of $(f_k,m\mbox{-pme})$ in $G_1$. On the other hand, 
\begin{equation}\label{comp11} 
\sup \limits_{I_{1/2}^\theta \times B_{1/2}^{}} |v_k(t,x) - \omega(t,x)| > \kappa_0,
\end{equation}
for any bounded weak solution $\omega$ of \eqref{m-eq} in $G^1_{\frac 12}$. By classical local H\"older estimates, $\{v_k\}$ is equicontinuous and bounded in $G^1_{\frac 12} \Subset G_1$. Therefore $v_k$ converges uniformly to some $\overline v$ which satisfies \eqref{m-eq} in $G^1_{\frac 12}$. This gives us a contradiction to \eqref{comp11} by choosing $k \gg 1$.
\end{proof}

\begin{remark}\label{rem}
Thanks to Theorem \ref{fbreg}, we observe there exists universal parameter $\tilde m>1$ and universal positive constant $C$ such that for any $\omega$ nonnegative bounded weak solution of \eqref{m-eq} with $1 < m \leq \tilde m$, there holds
\begin{equation}\label{fb}
\sup\limits_{I^{\theta'}_\rho \times B_{\rho}} \omega(t,x) \leq C\rho^{2},
\end{equation}
for $\theta'=2(1-m)+2$.
\end{remark}

First, we state the following proposition. 

\begin{proposition}\label{smallstep1} For the universal parameter $\tilde m > 1$ given in \eqref{fb},there exist positive universal constants $C_0, \varepsilon_0$ and $r_0$, such that if $v$ is a nonnegative bounded weak solution of \eqref{mf-eq} in $G_1$ such that
$$
1 < m \leq \tilde m \quad \mbox{and} \quad \|f\|_{\infty,G_1} \leq \varepsilon_0
$$
where $v$ satisfies
$$
0 \leq v(t,x) \leq 1 \quad \mbox{and} \quad v(0,0) \leq \frac 12 \, r_0^{\,\frac{2}{m}}, 
$$
then
$$
\sup \limits_{I_{r_0}^\theta \times B_{r_0}^{}} v(t,x) \leq r_0^{\,\frac{2}{m}}. 
$$ 
\end{proposition}
 
\begin{proof}
Hereafter in this section, we consider $v$ a nonnegative bounded weak solution of \eqref{mf-eq} for a fixed $m \in (1,\tilde m\,]$, as in Remark \ref{rem}. First, we show the existence of universal small parameter $r_0$, such that if
\begin{equation}\label{lips4}
\begin{array}{lcc}
v(0,0) \leq \frac 12\, r_0^{\frac{2}{m}} &
\quad \mbox{then} \quad 
 & \sup\limits_{I^\theta_{r_0} \times B_{r_0}} v(t,x) \leq r_0^{\frac{2}{m}}. 
\end{array}
\end{equation} 

Indeed, let $\kappa$ be as in Lemma \ref{complemma2} and $\varepsilon=\varepsilon(\kappa)>0$ such that the pointwise estimate \eqref{lips1} holds. In view of estimate \eqref{fb}, estimate \eqref{lips1} and $\theta' < \theta$, we have
\begin{equation}\nonumber 
\begin{array}{ccl}
\sup\limits_{I^\theta_{r} \times B_{r}} v(t,x) & \leq & \varepsilon + \sup\limits_{I^\theta_{r} \times B_{r}} \varpi(t,x) \\[0.5cm]
  & \leq & \varepsilon + \sup\limits_{I^{\theta'}_{r} \times B_{r}} \varpi(t,x) \\[0.5cm]
    & \leq & \varepsilon + Cr^{2}, \\
 \end{array} 
\end{equation}
for radii $0<r \leq 1/2$. Therefore, by making the following universal choices   
$$
r_0= \left(\frac{1}{2C}\right)^{\frac{m}{2(m-1)}}\quad \mbox{and} \quad \varepsilon=\frac{1}{2}r_0^{\,\frac{2}{m}}, 
$$
we obtain \eqref{lips4} as desired. 
\end{proof}

\begin{proposition}\label{small} For the universal parameter $\tilde m > 1$ given  in \eqref{fb},there exist positive universal constants $C_0, \varepsilon_0$ and $r_0$, such that if $v$ is a nonnegative bounded weak solution of \eqref{mf-eq} in $G_1$ such that
$$
1 < m \leq \tilde m \quad \mbox{and} \quad \|f\|_{\infty,G_1} \leq \varepsilon_0
$$
where $v$ satisfies
\begin{equation}\label{pointwise}
0 \leq v(t,x) \leq 1 \quad \mbox{and} \quad v(0,0) \leq \frac 12 \, r^{\frac{2}{m}}
\end{equation}
for each $0< r \leq r_0$, then
$$
\sup \limits_{I_{r}^\theta \times B_{r}^{}} v(t,x) \leq C_0\, r^{\frac{2}{m}}.
$$ 
\end{proposition}

\begin{proof} First, let us assume parameters $r_0$ and $\varepsilon_0$ as in Proposition \ref{smallstep1}. We want to show that if the pointwise estimate 
\begin{equation}\label{lips12}
v(0,0) \leq \frac 12 r_0^{\,k \frac{2}{m}}
\end{equation}
holds for a positive integer $k$, then
\begin{equation}\label{lips11}
\sup\limits_{I^\theta_{r_0^k} \times B_{r_0^k}} v(t,x) \leq r_0^{k\frac{2}{m}}. 
\end{equation} 
Indeed, we prove this by induction on $k$. Let us denote by $\mathcal{P}_k$ the following  
statement: if estimate \eqref{lips12} is pointwisely satisfied, then estimate \eqref{lips11} holds. Note that Proposition \ref{smallstep1} corresponds to the case $\mathcal{P}_1$. Now, we assume the statement $\mathcal{P}_k$ holds for $k>1$. Under the assumption \eqref{lips12} for $k+1$ instead of $k$ we consider
$$ 
\tilde{v}(x,t)=\frac{v(r_0^{\theta}t,r_0x)}{r_0^{\frac{2}{m}}}.
$$ 
Easily, we observe that $\tilde{v}$ is still a nonnegative bounded weak solution of $(\tilde f,m\mbox{-pme})$  for the parameter $1<m \leq \tilde m$ previously fixed and $\|\tilde f\|_\infty \leq \varepsilon_0$, such that
$$
\tilde{v}(0,0) \leq \frac 12 \, r_0^{\,k\frac{2}{m}}.
$$
Therefore, from statement $\mathcal{P}_{k}$, we have 
\begin{equation}\nonumber
\sup\limits_{I^\theta_{r_0^{\,k+1}} \times B_{r_0^{\,k+1}}} {v}(t,x) = \sup\limits_{I^\theta_{r_0^{\,k}} \times B_{r_0^{\,k}}} \tilde{v}(t,x)  \, r_0^\frac{2}{m} \leq r_0^{\,(k+1)\frac{2}{m}} 
\end{equation}
and so, $\mathcal{P}_{k+1}$ is obtained. Therefore $\mathcal{P}_{k}$ holds for every positive integer $k$. 

Finally we are ready to conclude the proof of Proposition \ref{small}. Fixed radius $0<r \leq r_0$, let us choose the integer $ k_r>0$ such that $r_0^{k_r+1} < r \leq r_0^{k_r}$. If
$$
v(0,0) \leq \frac 12 \, r^{\frac{2}{m}} \; \left(\leq \frac 12 \, r_0^{k_r\frac{2}{m}}\right) 
$$ 
then from $\mathcal{P}_k$, we obtain 
$$
\sup\limits_{I^\theta_{r} \times B_{r}} {v}(t,x) \leq \sup\limits_{I^\theta_{r_0^{\,k_r}} \times B_{r_0^{\,k_r}}} {v}(t,x) \leq r_0^{-\frac{2}{m}} r^\frac{2}{m}.
$$
\end{proof}

\begin{proof}[Proof of Theorem \ref{fbregnon}]
As observed in \eqref{norm} we can assume, with no loss of generality, that  nonnegative bounded solutions of \eqref{mf-eq}, are under the following regime 
$$
0 \leq v(t,x) \leq 1 \quad \mbox{with} \quad \|f\|_\infty \leq \varepsilon_0,
$$
for $\varepsilon_0$ as in Proposition \ref{small}. Using a covering argument, it is enough to prove estimate \eqref{fbestnon} for the particular case $t_0=0$ and $x_0=0$. Since $(0,0) \in \mathcal{F}(u)$ satisfies directly the pointwise estimate in \eqref{pointwise}, we have that the proof of Theorem \ref{fbregnon} follows easily as a particular case of Proposition \ref{small}.
\end{proof} 

\medskip

\begin{proof}[Proof of Theorem \ref{locreg}] 
It is enough to show estimate \eqref{locest} for the particular case $y=0$ and $s=0$. For this, we have to show there exist positive constants $m_0$ and $C$ such that, for a given nonnegative weak solution $v$ of $\eqref{mf-eq}$, with $0<m\leq m_0$, there holds
\begin{equation}\label{finalest}
\sup\limits_{I_r^{2}\times B_r} |v(t,x)-v(0,0)| \leq C r,
\end{equation}
for any $0<r\ll 1$. We also note that after a normalization argument, as argued in \eqref{norm}, we can assume with no loss of generality, that the solution $v$ satisfies the smallness conditions required in Proposition \ref{small}. For a universal parameter $r_0$ as in Proposition \ref{small}, let us denote 
$$
r_\star:=\left(2v(0,0)\right)^{\frac{m}{2}}. 
$$
We split the proof into three cases.

\medskip

\noindent
\textit{Case 1. $r_\star \leq r \leq r_0$}.  
For this, we have  
$$
v(0,0) \leq \frac 12 \,r^{\frac{2}{m}}.
$$
By Proposition \ref{small}, and assuming that  $1<m\leq \tilde m$, we derive
\begin{equation}\nonumber
\sup\limits_{I_r^{2}\times B_r} |v(t,x)-v(0,0)| \leq \sup\limits_{I_r^{\theta}\times B_r} v(t,x) \leq C\, r^{\frac{2}{m}}  \leq C r
\end{equation}
and estimate \eqref{finalest} follows.

\medskip
 
\noindent
\textit{Case 2. $0<r<r_\star \leq r_0$}.  
In this case, we define the rescaled function
$$
\overline{v}(t,x):= \dfrac{v(r_\star^{\theta} t,r_\star x)}{r_\star^{\frac{2}{m}}} \quad \mbox{in} \quad G_1,
$$
for $\theta=2/m$. We note that $\overline{v}$ is a nonnegative bounded weak solution of
$$
\overline v_t - div(m \,\overline v^{\,m-1} \nabla \overline v) = \overline f \quad \mbox{in } G_1,
$$
with $\| \overline f \|_\infty =  \| f \|_\infty$, satisfying
$$
\overline{v}(0,0)=\frac 12 \quad \mbox{and} \quad \sup\limits_{G_1} \overline{v}(t,x) \leq C,
$$  
where the last estimate is obtained by using Proposition \ref{small} for the radius $r_\star$. This makes $\overline{v}$ a function universally continuous and so, it is possible to find a universal parameter $\tau_0>0$ such that
$$
\overline{v}(t,x) > \frac{1}{4} \quad \mbox{for} \quad (t,x) \in I^{\theta}_{\tau_0} \times B_{\tau_0}. 
$$  
In view of this, $\overline{v}$ solves a uniformly parabolic equation of the form 
\begin{equation}\label{unifparab}
\overline v_t - div(\mathcal{A}(t,x) \nabla \overline v) = \overline f \quad  \mbox{in } I^{\theta}_{\tau_0} \times B_{\tau_0},
\end{equation}
for  some continuous coefficient $\lambda \leq \mathcal{A}(t,x) \leq \Lambda$ with universal ellipticity constants $0<\lambda < \Lambda$. By classical regularity estimates, we have $\overline v$ satisfies 
$$
\sup\limits_{I_r^{2}\times B_r} |\overline v(t,x)- \overline v(0,0)| \leq C r,
$$
for any $r>0$ such that  $I^2_r \subset I^2_{\tau_0} \Subset I^\theta_{\tau_0} $. Therefore 
$$
\sup\limits_{I_{rr_\star}^{2}\times B_{rr_\star}} |v(t,x)- v(0,0)| \leq C r r_\star^{\frac{2}{m}} \leq C rr_\star
$$
and so we guarantee that $v$ satisfies estimate \eqref{finalest} for any $0<r \leq r_\star\tau_0$. In order to conclude the argument of this case, we have to prove estimate \eqref{finalest} also holds for radii $r_\star\tau_0 < r \leq r_\star$. Indeed, using Proposition \ref{small} for radius $r_\star$ once more, we get
$$
\sup\limits_{I_{r}^{2}\times B_{r}} |v(t,x)- v(0,0)| \leq \sup\limits_{I_{r_\star}^{2}\times B_{r_\star}} |v(t,x)- v(0,0)|  \leq C r_\star \leq \frac{C}{\tau_0} r.
$$

\medskip

\noindent
\textit{Case 3. $r_\star > r_0$}. We easily observe that
$$
v(0,0)> \frac 12 r_0^{\frac{2}{m}}.
$$
By continuity, $v(t,x)$ is universally bounded from below in a small universal cylinder centered at $(0,0)$. Therefore, $v$ solves a uniformly parabolic equation as in \eqref{unifparab}, which implies that,  for some universal $r_1>0$, $v$  satisfies estimate \eqref{finalest} for any $0<r\leq r_1$.     

\medskip

Finally, we conclude estimate \eqref{finalest} holds for radii $0<r\leq \min\{r_0,r_1\}$. 

\end{proof}

\bigskip

{\small \noindent{\bf Acknowledgments.} The author is partially supported by CNPq 427070/2016-3 and grant 2019/0014 Paraiba State Research Foundation (FAPESQ). The author would like to thank the hospitality of the Abdus Salam International Centre for Theoretical Physics (ICTP), where parts of this work were conducted.

\bigskip

\bibliographystyle{amsplain, amsalpha}

\end{document}